\documentclass{amsart}
\usepackage{amsfonts, amssymb}
\usepackage{amsmath}
\usepackage{amsthm}

\numberwithin{equation}{section}

\newtheorem{theorem}{Theorem}[section]
\newtheorem{lemma}[theorem]{Lemma}

\newtheorem{proposition}[theorem]{Proposition}

\newcommand{\N}{\mathbb{N}}

\begin{document}

\title{A lemma on the difference quotients}

\author{Risto Korhonen}
\address{Department of Physics and Mathematics, University of Eastern Finland, P.O. Box 111,
FI-80101 Joensuu, Finland}
\email{risto.korhonen@uef.fi}
\thanks{The first author is supported in part by the Academy of Finland grant (\#286877) and (\#268009), and the second author is supported by the JSPS KAKENHI Grant Number JP16K05194, the third author is supported by the China Scholarship Council (CSC)}

\author{Kazuya Tohge}
\address{College of Science and Engineering, Kanazawa University, Kakuma-machi, Kanazawa, 920-1192, Japan}
\email{tohge@se.kanazawa-u.ac.jp}
%

\author{Yueyang Zhang}
\address{Department of Physics and Mathematics, University of Eastern Finland, P.O. Box 111,
FI-80101 Joensuu, Finland}
\email{yueyang.zhang@uef.fi}

\author{Jianhua Zheng}
\address{Department of mathematical sciences, Tsinghua University, Beijing 100084, China}
\email{jzheng@math.tsinghua.edu.cn}


\date{\today}

\commby{}

\begin{abstract}
Using a new Borel type growth lemma, we extend the difference analogue of the lemma on the logarithmic derivative due to Halburd and Korhonen to the case of meromorphic functions $f(z)$ such that $\log T(r,f)\leq r/(\log r)^{2+\nu}$, $\nu>0$, for all sufficiently large $r$. The method by Halburd and Korhonen
implies an estimate for the lemma on difference quotients, where the exceptional set is of finite logarithmic measure. We show the necessity of this set by proving that it must be of infinite linear measure for meromorphic functions whose deficiency is dependent on the choice of the origin. In addition, we show that there is an infinite sequence of $r$ in the set for which $m(r,f(z+c)/f(z))$ is not small compared to $T(r,f)$ for entire functions constructed by Miles. We also give a discrete version of Borel type growth lemma and use it to extend Halburd's result on first order discrete equations of Malmuist type.
\end{abstract}

\maketitle



\section{Introduction}

The lemma on the logarithmic derivatives is one of the key results needed in proving Nevanlinna's second main theorem \cite{nevanlinna:25}, as well as an important tool in analyzing value distribution of entire and meromorphic solutions of differential equations \cite{laine:93}. Nevanlinna theoretic approach by Ablowitz, Halburd and Herbst \cite{ablowitzhh:00} to study difference Painlev\'e equations leads to a need of finding extensions of value distribution theory for difference operators. A lemma on difference quotients for finite order meromorphic functions was introduced in two independent studies, by Halburd and the first author \cite{halburdk:06JMAA,halburdk:06AASFM}, and by Chiang and Feng \cite{chiangf:08}. The lemma on difference quotients was later on extended to include meromorphic functions $f$ of hyper-order less than one by Halburd and the first and the second author \cite{halburdkt:14TAMS} as
    \begin{equation}\label{hkt_est}
    m\left(r,\frac{f(z+c)}{f(z)}\right) = o\left(\frac{T(r,f)}{r^{1-\varsigma-\varepsilon}}\right),
    \end{equation}
where $\varepsilon >0$, $c\in\mathbb{C}\setminus\{0\}$ and $r\to\infty$ outside of a set of finite logarithmic measure, and the \emph{hyper-order}~$\varsigma$ of $f$ is defined as
    \begin{equation*}
    \varsigma=\varsigma(f)=\limsup_{r\to\infty} \frac{\log\log T(r,f)}{\log r}.
    \end{equation*}
The estimate by Chiang and Feng \cite{chiangf:08} can be written in the form
    \begin{equation}\label{edmund_est}
    m\left(r,\frac{f(z+c)}{f(z)}\right) = O(r^{\sigma-1+\varepsilon}),
    \end{equation}
where $r\to\infty$ without an exceptional set and the \emph{order} $\sigma$ of $f$ is defined as
\begin{equation*}
\sigma=\sigma(f)=\limsup_{r\to\infty}\frac{\log T(r,f)}{\log{r}}.
\end{equation*}
Chiang and Feng \cite{chiangf:09,chiangf:16} and, independently, Bergweiler and Langley \cite{bergweilerl:07} have obtained Wiman-Valiron type estimates for difference quotients in the case of order $<1$ meromorphic functions. These results extend a Wiman-Valiron method for differences due to Ishizaki and Yanagihara \cite{ishizakiy:04}. Recently, Chiang and Feng \cite{chiangfengAW} obtained the Askey-Wilson logarithmic difference estimate for meromorphic function of finite logarithmic order, and Cheng and Chiang \cite{chengchiang17} proved a lemma on the logarithmic Wilson differences for finite order meromorphic functions.

The necessity of the appearance of the exceptional set in the lemma on difference quotients has so far remained an open question. By a superficial observation, in estimates of the type \eqref{edmund_est} the exceptional set does not seem to be present. However, for functions of irregular growth whose \emph{lower order} $\lambda$ is defined by
\begin{equation*}
\lambda=\lambda(f)=\liminf_{r\to\infty}\frac{\log T(r,f)}{\log{r}},
\end{equation*}
and satisfies $\lambda \leq \sigma-1$, the error term on the right hand side of \eqref{edmund_est} is, in fact, bigger than the characteristic $T(r,f)$ for large part of the positive real line.
In this paper we will show, for meromorphic functions whose deficiency are dependent on the choice of the origin, such as functions constructed by Gol'dberg \cite{BelinskiuiGol:54} and Miles \cite{miles:83}, that the exceptional set in the estimate \eqref{hkt_est} cannot be reduced into a set of finite linear measure.
In particular, we will show that for Miles' example of entire functions $f$ with order $3/2<\sigma<\infty$, there is an infinite sequence of $r$ in the exceptional set such that
    \begin{equation*}
    m\left( r,\frac{f(z+c)}{f(z)}\right)\not= o(T(r,f)),
    \end{equation*}
as $r\to\infty$. Before this, in Section~\ref{diffquo_sec} we will first extend the lemma on difference quotients to a slightly more general case where $f$ is a meromorphic function such that $\log T(r,f)\leq r/(\log r)^{2+\nu}$ for any $\nu> 0$ and for all $r$ sufficiently large by showing that for such functions
    \begin{equation}\label{difana_intro}
    m\left( r,\frac{f(z+c)}{f(z)}\right)= o\left(\frac{T(r,f)}{(\log r)^{\nu-\varepsilon}}\right),
    \end{equation}
for all $r$ outside of a set of finite logarithmic measure. A key tool in the proof of this extension is a new Borel type growth lemma, introduced in Section~\ref{borel_sec} below. A purely discrete form of the Borel lemma is given in Section~\ref{discrete_sec}, which extends a result due to Al-Ghassani and Halburd \cite{al-ghassanih:15}. Finally, we obtain a Malmquist type theorem for first order discrete equations in Section~\ref{dMalm_sec} as an application of the discrete Borel lemma.






\section{Borel type growth lemma extensions}\label{borel_sec}

The following lemma is an extension of the growth lemma \cite[Lemma~8.3]{halburdkt:14TAMS}.

\begin{lemma}\label{technical}
Let $T:[0,\infty)\to(0,\infty)$ be a non-decreasing continuous
function and let $s\in(0,\infty)$. If
    \begin{equation}\label{assu}
    \limsup_{r\to\infty}\frac{h(r)h(rh(r)) \log T(r)}{r}=\zeta,
    \end{equation}
where $\zeta\in[0,\infty)$ and $h:[r_0,\infty)\to(0,\infty)$ is an increasing function such that
    \begin{equation*}
    \int_{r_0}^\infty \frac{1}{th(t)}\,dt
    \end{equation*}
converges, then
   \begin{equation}\label{concl}
    T(r+s) = T(r)+(\zeta+o(1))\left(\frac{T(r)}{h(r)}\right),
    \end{equation}
where $r$ runs to infinity outside of a set $E$ of finite logarithmic measure, i.e. $\int_{E}dr/r<\infty$.
\end{lemma}

\begin{proof}
For a fixed constant $\eta\in\mathbb{R}^+$ such that $\eta>\zeta$, assume that the set $F_{\eta}\subset[1,\infty)$ defined by
    \begin{equation}\label{Fdef}
    F_{\eta}=\left\{r\in\mathbb{R}^+:
    \frac{T(r+s)-T(r)}{T(r)}\cdot h(r)\geq \eta
    \right\}
    \end{equation}
is of infinite logarithmic measure. 
Note that $F_\eta$ is a closed
set and therefore it has a smallest element, say $r_0$. Set
$r_n=\min(F_\eta\cap [r_{n-1}+s,\infty))$ for all $n\in\mathbb{N}$. Then
the sequence $\{r_n\}_{n\in\mathbb{Z}^+}$ satisfies $r_{n+1}-r_n\geq s$
for all $n\in\mathbb{Z}^+$, $F_\eta\subset \bigcup_{n=0}^\infty
[r_n,r_n+s]$ and
    \begin{equation}\label{assuinpr2}
    \left(1+\frac{\eta}{h(r_n)}\right)T(r_n)\leq T(r_{n+1})
    \end{equation}
for all $n\in\mathbb{Z}^+$.

Suppose that there exists an $m\in\mathbb{Z}^+$
such that $r_n\geq n h(n)$ for all $r_n\geq m$. But
then,
    \begin{eqnarray*}
    \int_{F_\eta\cap[1,\infty)}\frac{dt}{t} &\leq& \sum_{n=0}^\infty
    \int_{r_n}^{r_n+s}\frac{dt}{t}
    \leq \int_1^{m} \frac{dt}{t} +  \sum_{n=1}^\infty
    \log\left(1+\frac{s}{r_n}\right)\\
    &\leq& \sum_{n=1}^\infty
    \log\left(1+\frac{s}{nh(n)} \right) +O(1)  <\infty,
    \end{eqnarray*}
which contradicts the assumption $\int_{F_\eta\cap[1,\infty)}dt/t=\infty$.
Therefore, the sequence $\{r_n\}_{n\in\mathbb{Z}^+}$ has a subsequence
$\{r_{n_j}\}_{j\in\mathbb{Z}^+}$ such that $r_{n_j}<
n_j h(n_j)$ for all $j\in\mathbb{Z}^+$.
Since $r_{n+1}-r_n\geq s$, we may assume without loss of generality that $s\geq1$, by taking another subsequence of $\{r_{n_j}\}_{j\in\mathbb{Z}^+}$ if necessary. By iterating~\eqref{assuinpr2} along the sequence
$\{r_{n_j}\}_{j\in\mathbb{Z}^+}$, we have
    \begin{equation*}
    T(r_{n_j})\geq \prod_{\nu=0}^{n_j-1}\left(1+\frac{\eta}{h(r_\nu)}\right)T(r_0)
    \end{equation*}
for all $j\in \mathbb{Z}^+$. It follows that
    \begin{equation*}
    r_{n_j} \geq r_0 + n_js \geq n_js \geq n_j
    \end{equation*}
for all $j\in\mathbb{Z}^+$, and so
    \begin{equation*}
    \begin{split}
    & \limsup_{r\to\infty}\frac{h(r)h(rh(r))\log T(r)}{r} \\
    &\quad \geq
    \limsup_{j\to\infty}\frac{\displaystyle h(r_{n_j})h(r_{n_j} h(r_{n_j}))\left(\log T(r_0)+\sum_{\nu=0}^{n_j-1}\log\left(1+\frac{\eta}{h(r_\nu)}\right)\right)}{r_{n_j}}\\
    &\quad \geq\limsup_{j\to\infty}\frac{\displaystyle h(n_j)h(n_jh(n_j))\left(\log
    T(r_0)+n_j\log\left(1+\frac{\eta}{h(r_{n_j})}\right)\right)}
    {n_j h(n_j)}
    \\
    & \quad \geq  \limsup_{j\to\infty}\frac{\displaystyle h(n_j)h(n_jh(n_j)) \Bigg(\log T(r_0)+n_j
    \frac{\eta}{h(n_jh(n_j))}
    \log\Bigg(1+\frac{\eta}{h(n_jh(n_j))}\Bigg)^{\frac{h(n_jh(n_j))}{\eta}}\Bigg)}
    {n_jh(n_j)}
    \\
    & \quad = \limsup_{j\to\infty} \frac{\eta\cdot n_j\cdot 1 \cdot h(n_j)h(n_jh(n_j)) }{n_jh(n_j)h(n_jh(n_j))} = \eta >\zeta,
    \end{split}
    \end{equation*}
which contradicts \eqref{assu}. Hence the logarithmic measure of
$F_\eta$ defined by \eqref{Fdef} must be finite, and so
    \begin{equation*}
    T(r+s) = T(r)+ (\zeta+o(1))\left(\frac{T(r)}{h(r)}\right)
    \end{equation*}
for all $r$ outside of a set of finite logarithmic measure.
Therefore the assertion \eqref{concl} follows.
\end{proof}

Let $\zeta(r)$ be a function of $r$ which has a finite limit as $r\rightarrow\infty$ and $f(z)$ be a meromorphic function. Under the assumptions of Lemma~\ref{technical}, if $T(r,f)$, $r\geq r_0>e$, grows in such a way that
\begin{equation}\label{assuinpr3}
\log T(r,f) \leq \frac{r\zeta(r)}{h(r)h(rh(r))},
\end{equation}
then $T(r,f)$ satisfies the asymptotic relationship \eqref{concl}. However, when $\zeta(r)$ is a function of $r$ which tends to $\infty$ as $r\rightarrow\infty$, the situation will be different. For example, for the function $f(z)=\exp(e^{z})$, it follows from \cite[p.~7]{hayman:64} that
\begin{equation*}
\log T(r,f) \sim r-\frac{1}{2}\log r- \frac{1}{2}\log(2\pi^3).
\end{equation*}
We see that in this case $\zeta=\infty$ in \eqref{assu}, but $T(r+1,f)=(e+o(1))T(r,f)$. This example shows that Lemma~\ref{technical} is best possible when applied to meromorphic functions in the sense that $\zeta$ cannot be extended to $\infty$ in general.

\section{Lemma on the difference quotients}\label{diffquo_sec}

Let $c$ be a nonzero constant and $f(z)$ be a nonconstant meromorphic function. It is shown in \cite[p.~66]{goldbergo:70} (see also \cite{ablowitzhh:00} or \cite{yanagihara:80}) that the following double inequality
\begin{equation}\label{ttt rr0}
(1+o(1))T(r-|c|,f(z)) \leq T(r,f(z+c))\leq (1+o(1))T(r+|c|,f(z)),
\end{equation}
holds as $r\rightarrow \infty$ without any exceptional set. If $T(r,f)$ also satisfies \eqref{assuinpr3}, then for a given small $\varepsilon>0$ and any number $v>2\varepsilon$, we may choose $h(r)$ in \eqref{assuinpr3} to be $h(r)=(\log r)^{1+\varepsilon}$ and let $\zeta(r)=h(r)h(rh(r))/(\log r)^{2+v}$, where $r\geq r_0>e$.
Now, the inequality \eqref{assuinpr3} becomes $\log T(r,f) \leq r/(\log r)^{\nu+2}$ for all $r\in [r_0,\infty)$
and, by Lemma~\ref{technical}, it follows that for any finite number $s>0$,
    \begin{equation}\label{ttt r}
    T(r+s,f) = T(r,f)+ o\left(\frac{T(r,f)}{(\log r)^{1+\varepsilon}}\right),
    \end{equation}
for all $r$ outside of a set of finite logarithmic measure. By combining \eqref{ttt rr0} and \eqref{ttt r}, we get the asymptotic relation
\begin{equation}\label{ttt rr1}
T(r,f(z+c))= (1+o(1))T(r,f(z)),
\end{equation}
where $r\rightarrow \infty$ outside an exceptional set of finite logarithmic measure.
Similarly to \eqref{ttt rr0}, it is also suggested in \cite[p.~66]{goldbergo:70} that for any nonconstant meromorphic function $f(z)$ the double inequality
\begin{equation}\label{ttt rr00}
(1+o(1))N(r-|c|,f(z)) \leq N(r,f(z+c))\leq (1+o(1))N(r+|c|,f(z)),
\end{equation}
holds as $r\rightarrow \infty$. Since the integrated counting function $N(r,f)$ is nondecreasing, continuous and satisfies $N(r,f) \leq T(r,f)$, when
$f$ satisfies $\log T(r,f)\leq r/(\log r)^{2+\nu}$, $\nu> 0$, we have by Lemma~\ref{technical} and \eqref{ttt rr00} that
    \begin{equation}\label{ttt rr11}
    N(r,f(z+c)) = (1+o(1))N(r,f(z)),
   \end{equation}
holds as $r\rightarrow \infty$ outside an exceptional set of finite logarithmic measure. Thus the asymptotic  relation \eqref{ttt rr1} can also be obtained by combining \eqref{ttt rr11} and the following Lemma~\ref{danlogue}, which is an extension of the estimate on difference quotients \eqref{hkt_est}.

\begin{lemma}\label{danlogue}
Let $f$ be a nonconstant meromorphic function such that $\log T(r,f)\leq r/(\log r)^{2+\nu}$, $\nu>0$, for all sufficiently large $r$. Then, for a given $0<\varepsilon < \nu$,
    \begin{equation}\label{difana1}
    m\left( r,\frac{f(z+c)}{f(z)}\right)= o\left(\frac{T(r,f)}{(\log r)^{\nu-\varepsilon}}\right),
    \end{equation}
for all $r$ outside of an exceptional set with finite logarithmic measure.
\end{lemma}

\begin{proof}
Let $C>1$ and $r_0$ be such that $T(r,f) \geq x_0 >e$ for all $r\geq r_0$. By the Borel's lemma (see, e.g. \cite{cherryy:01}), we know that there is a positive, non-decreasing and continuous function $\xi(x)$, $x_0 \leq x<\infty$ such that the following inequality
\begin{equation*}
T\left(r+\frac{r}{\xi(T(r,f))},f\right)\leq CT(r,f)
\end{equation*}
holds for all $r$ outside a set $E$ satisfying
    \begin{equation}\label{size  E}
    \int_{E\cap[r_0,R]}\frac{dr}{r} \leq \frac{1}{\log C} \int_{x_0}^{T(R,f)}\frac{dx}{x \xi(x)}+O(1),
    \end{equation}
where $R<\infty$. For a given $0<\varepsilon < \nu$, let $\xi(x)=x^{\varepsilon}$ when $f$ has finite order and $\xi(x)=(\log x)(\log\log x)^{1+\varepsilon}$ when $f$ has infinite order. It follows from \eqref{size  E} that the following closed set
\begin{equation}\label{def  E}
E=\left\{r:T\left(r+|c|+\frac{r+|c|}{\xi(T(r+|c|,f))},f\right)\geq CT(r+|c|,f)\right\}
\end{equation}
has finite logarithmic measure at most. Let
\begin{equation*}
\alpha = 1+\frac{1}{\xi(T(r+|c|,f))}.
\end{equation*}
Then,
    \begin{equation}\label{exceptional set}
    T(\alpha(r+|c|),f) = T\left(r+|c|+\frac{r+|c|}{\xi(T(r+|c|,f))},f\right) \leq  CT(r+|c|,f)
    \end{equation}
holds for all $r$ outside of the set $E$. By assumption $\log T(r,f)\leq r/(\log r)^{2+\nu}$, $\nu> 0$, for all $r\geq r_0$ and so we have from \cite[Lemma~8.2]{halburdkt:14TAMS} that
    \begin{equation*}\label{difana0}
    \begin{split}
m\left( r,\frac{f(z+c)}{f(z)}\right)&= O\left(\frac{[\log T(r+|c|,f)][\log\log T(r+|c|,f)]^{1+\varepsilon}}{\delta(1-\delta)r^{\delta}}T(r+|c|,f)\right)\\
    &=O\left(\frac{(r+|c|)[\log(r+|c|)-(2+\nu)\log\log(r+|c|)]^{1+\varepsilon}}{[\log(r+|c|)]^{2+\nu}\delta(1-\delta)r^{\delta}}T(r+|c|,f)\right),
    \end{split}
    \end{equation*}
which together with \eqref{ttt r} yields the estimate \eqref{difana1} by choosing $\delta=1-1/\log r$.
\end{proof}

For meromoprhic functions of hyper-order $\varsigma\geq1$, the quantity $m(r,f(z+c)/f(z))$ can grow as fast as $T(r,f)$ for all $r\in [0,\infty)$, for example, for $f(z)=\exp(e^{z})$. But for meromorphic functions of hyper-order $\varsigma<1$, this case occurs only when the sequence of $r$ is in a certain set of finite logarithmic measure, as indicated by the estimate \eqref{hkt_est} or by \eqref{difana1}.
Note that when $\varsigma<1$ we need to choose $\xi(x)=(\log x)^{1+\varepsilon/3}$ in Lemma~\ref{danlogue} to get the estimate \eqref{hkt_est} when $f$ has infinite order as in the proof of \cite[Lemma~9.1]{halburdkt:14TAMS} and, consequently, the set $E$ in \eqref{def  E} is redefined to be
\begin{equation}\label{def  E00}
E=\left\{r:T\left(r+|c|+\frac{r+|c|}{(\log T(r+|c|,f))^{1+\varepsilon/3}},f\right)\geq CT(r+|c|,f)\right\},
\end{equation}
where $(\log T(r+|c|,f))^{1+\varepsilon/3} \leq (r+|c|)^{(\varsigma+\varepsilon/3)(1+\varepsilon/3)}$ and $(\varsigma+\varepsilon/3)(1+\varepsilon/3)<1$ for a sufficiently small $\varepsilon>0$ and large $r$.
Below we consider the necessity of the exceptional set which appears in the estimate \eqref{hkt_est} and the irregular behavior of the proximity function $m(r,f(z+c)/f(z))$ in the exceptional set. To this end, we first prove the following Proposition~\ref{necessity e}, which is a counterpart of a result due to Valiron \cite{valiron:47} (see also \cite[p.~271]{nevanlinna:70}) concerning the dependence on the choice of the origin of the deficiency of a general meromorphic function $f$. Recall that the \emph{deficiency} $d(0,f)$ for the value~$0$ is defined as
     \begin{equation*}
    d(0,f):=1-\limsup_{r\rightarrow\infty}\frac{N(r,1/f)}{T(r,f)}=\liminf_{r\rightarrow\infty}\frac{m(r,1/f)}{T(r,f)}.
    \end{equation*}
Valiron \cite{valiron:47} proved: If the characteristic function $T(r,f)$ of a meromorphic function $f$ satisfies the condition
    \begin{equation}\label{ttt r0}
    \lim_{r\rightarrow\infty}\frac{T(r+1,f)}{T(r,f)}=1,
    \end{equation}
then the deficiency of $f$ is independent on the choice of origin, that is, $d(0,f(z))=d(0,f(z+c))$ for any finite non-zero constant $c$.
The function $f(z)=\exp(e^z)$ which has~0 as the Picard exceptional value shows that the converse of Valiron's result is in general not true. We prove

\begin{proposition}\label{necessity e}
Let $f$ be a non-constant meromorphic function. If $d(0,f(z))>d(0,f(z+c))$ for some $c\not=0$, then there exists a constant $1<C<\infty$ and a set $E_0$ with infinite linear measure such that
    \begin{equation*}
     T(r+|c|,f)\geq CT(r,f),
    \end{equation*}
holds for all $r\in E_0$. Moreover, if $f$ is entire, then there is an infinite sequence of $r$ such that as $r\rightarrow\infty$,
    \begin{equation*}
     m\left(r,\frac{f(z+c)}{f(z)}\right)\not=o(T(r,f)).
    \end{equation*}
\end{proposition}

A simple counterexample for Proposition~\ref{necessity e} is the function $f(z)=e^z$ with characteristic $T(r,f)=r/\pi$. This function satisfies $d(0,f(z))=d(0,f(z+c))=1$ for any nonzero constant $c$ while $T(r+|c|,f)=T(r,f)+O(1)$ and $m(r,f(z+c)/f(z))=O(1)$ for all $r\geq 0$.
On the other hand, meromorphic functions having the property $d(0,f(z))\not=d(0,f(z+c))$ for some constant $c$ do exist, see Miles \cite{miles:83}. For the finite order case, we recall the following two examples. The first one is due to Gol'dberg \cite{BelinskiuiGol:54} who constructed a meromorphic function with order~1 such that
     \begin{equation*}
     \begin{split}
    &d(0,f(z))=1, \quad \text{and}\\
    &d(0,f(z+c))=0, \quad \text{for some $c\not=0$}.
    \end{split}
    \end{equation*}
The second one is due to Miles \cite{miles:83} who proved: There exists an entire function $f$ of order $3/2<\sigma(f)<\infty$ such that
     \begin{equation*}\label{ddependence}
     \begin{split}
     &d(0,f(z))=0,  \quad \text{and}\\
     &d(0,f(z+c))\geq \rho>0, \quad \text{for all $c\not=0$},
     \end{split}
    \end{equation*}
for some $\rho<1$ independent of $c$.


\begin{proof}[Proof of Proposition~\ref{necessity e}]
For simplicity, denote $d_1=d(0,f(z))$ and $d_2=d(0,f(z+c))$. Then $0\leq d_2<d_1\leq 1$. Note that the relation \eqref{ttt rr00} also holds for the counting function $N(r,1/f)$. By this fact and \eqref{ttt rr0}, we get
    \begin{equation*}
    \limsup_{r\rightarrow\infty}\frac{N(r,1/f(z+c))}{T(r,f(z+c))}\leq \limsup_{r\rightarrow\infty}\frac{N(r+|c|,1/f(z))}{T(r+|c|,f(z))}\cdot\frac{T(r+|c|,f(z))}{T(r-|c|,f(z))}.
    \end{equation*}
It follows that
    \begin{equation}\label{ttt r01a}
    \limsup_{r\rightarrow\infty}\frac{T(r+|c|,f)}{T(r-|c|,f)}=\infty
    \end{equation}
in the case $d_1=1$ and
    \begin{equation}\label{ttt r01b}
    \limsup_{r\rightarrow\infty}\frac{T(r+|c|,f)}{T(r-|c|,f)}\geq \frac{1-d_2}{1-d_1}
    \end{equation}
in the case $d_1<1$. Let $C$ be a real constant such that $1<C<\infty$ when $d_1=1$ and $1<C^3<(1-d_2)/(1-d_1)$ when $d_1<1$. Define the set
\begin{equation}\label{ttt r02}
E_0=\{r:T(r+|c|,f)\geq CT(r,f)\}.
\end{equation}
We claim that this set has infinite linear measure. Otherwise, the inequality
    \begin{equation}\label{ttt r1}
    T(r+|c|,f) \leq CT(r,f)
    \end{equation}
holds for all $r$ outside an exceptional set with finite linear measure. Recall the following lemma from \cite[Lemma~3.1]{halburdk:10JAMS}: Let $\mu$ be a positive, strictly increasing differentiable function of $r$ defined on $(r_0,\infty)$ for some $r_0$ and let $g_1(r)$ and $g_2(r)$ be two nondecreasing functions for all $r_0<r<\infty$ such that $g_1(r)\leq g_2(r)$ for all $r\in (r_0,\infty)\setminus E_1$, where the exceptional set $E_1$ satisfies
    \begin{equation*}
    \int_{t\in E_1\cap[r_0,\infty)}d\mu(t)<\infty.
    \end{equation*}
Then, for a given $\epsilon>0$, there is an $\hat{r}\geq r_0$ such that $g_1(r)\leq g_2(s(r))$ for all $r\geq \hat{r}$, where $s(r)=\mu^{-1}(\mu(r)+\epsilon)$. By applying this lemma with $\mu(r)=r$ to \eqref{ttt r1} we obtain, for a given $\epsilon$ such that $0<3\epsilon \leq |c|$, there is a large enough $r_0$ such that for all $r\in [r_0,\infty)$,
    \begin{equation*}
    T(r+|c|,f)\leq CT(r+\epsilon,f),
    \end{equation*}
which implies that the superior limit of $T(r+|c|,f)/T(r+\epsilon,f)$, as $r$ approaches $\infty$, is at most $C$. But then
     \begin{equation*}
    \limsup_{r\rightarrow\infty}\frac{T(r+|c|,f)}{T(r-|c|,f)}\leq \limsup_{r\rightarrow\infty}\frac{T(r+|c|,f)}{T(r+\epsilon,f)}\cdot\frac{T(r+\epsilon,f)}{T(r-|c|+2\epsilon,f)}\cdot\frac{T(r-|c|+2\epsilon,f)}{T(r-|c|,f)}\leq C^3,
    \end{equation*}
a contradiction to \eqref{ttt r01a} or \eqref{ttt r01b}. Hence the set defined in \eqref{ttt r02} must be of infinite linear measure.

If $f$ is an entire function such that $d_1>d_2$, then we can deduce that
        \begin{equation}\label{NN 1}
        \begin{split}
    T(r,f(z+c))&=m(r,f(z+c))\leq m(r,f(z))+m\left(r,\frac{f(z+c)}{f(z)}\right)\\
    &=T(r,f(z))+m\left(r,\frac{f(z+c)}{f(z)}\right).
    \end{split}
    \end{equation}
On the other hand, by the definition of deficiency, for a given $\varepsilon$ satisfying $0<2\varepsilon<d_1-d_2$, we also have
    \begin{equation}\label{NN 2}
    (d_1-\varepsilon)T(r,f(z))\leq m\left(r,\frac{1}{f(z)}\right) \leq  m\left(r,\frac{1}{f(z+c)}\right)+m\left(r,\frac{f(z+c)}{f(z)}\right)
    \end{equation}
holds for all sufficiently large $r$.
From the definition of $d_2$, it follows that there is a sequence of $r$ such that $m(r,1/f(z+c))\leq (d_2+\varepsilon) T(r,f(z+c))$ when $r$ is sufficiently large. Together with this inequality, we get from \eqref{NN 1} and \eqref{NN 2} that
        \begin{equation}\label{NN 3}
m\left(r,\frac{f(z+c)}{f(z)}\right)\geq \frac{d_1-d_2-2\varepsilon }{1+d_2+\varepsilon}\cdot T(r,f(z))
    \end{equation}
holds for all sufficiently large $r$ in this sequence. Thus our assertion follows.

\end{proof}


For a meromorphic function $f$ whose deficiency is dependent on the choice of origin, if $T(r,f)$ satisfies the condition \eqref{assu}, then by Lemma~\ref{technical} the estimate \eqref{concl} holds for all $r$ outside the set defined in \eqref{Fdef} with finite logarithmic measure. Since the constant $C$ in Proposition~\ref{necessity e} satisfies $C\geq 1+\eta/h(r)$ when $r$ is sufficiently large, it follows that the set in \eqref{Fdef} is also of infinite linear measure. However, from the proof of Lemma~\ref{danlogue}, it is seen that the exceptional set in the estimate \eqref{difana1} (also in~\eqref{hkt_est}) is independent of the one in \eqref{Fdef} whenever the characteristic function $T(r,f)$ satisfies
    \begin{equation*}
    \limsup_{r\rightarrow\infty}\frac{T(r+|c|,f)}{T(r,f)}<\infty,
    \end{equation*}
and, from the proof of Proposition~\ref{necessity e}, we see that this case cannot be excluded automatically when the deficiencies also satisfy $\max\{d(0,f(z)),d(0,f(z+c))\}<1$. Miles' function above is such an example.

After the above remark, we now apply Proposition~\ref{necessity e} to study the necessity of the exceptional set in \eqref{hkt_est}. Consider a meromorphic function $f$ of hyper-order $\varsigma<1$ whose deficiency is dependent on the choice of origin. Firstly, from Proposition~\ref{necessity e} it follows that for some finite constant $C>1$ the set $E_2=\{r:T(r+2|c|,f)\geq CT(r+|c|,f)\}$ is of infinite linear measure. From the definition of $E$ in \eqref{def  E00}, it is seen that there is a sufficiently large $r_0$ such that $E_2\cap[r_0,\infty) \subseteq E\cap[r_0,\infty)$. Thus the exceptional set associated with the error term in the estimate \eqref{hkt_est} is also of infinite linear measure. Secondly, if $f$ is entire, from Proposition~\ref{necessity e} we know that there is an infinite sequence of $r$ such that $m(r,f(z+c)/f(z))\not=o(T(r,f))$ as $r\rightarrow\infty$. By \eqref{hkt_est} it follows that
the infinite sequence of $r$ satisfying \eqref{NN 3} must be in the exceptional set associated with the error term in \eqref{hkt_est}.
In conclusion, for entire functions of hyper-order less than~1 whose deficiency is dependent on the choice of origin, the estimate \eqref{hkt_est} is not applicable in the exceptional set with infinite linear measure. This answers the question of the necessity of the exceptional set in the lemma on difference quotients, posed in the introduction.

As pointed out by Chiang and Luo \cite[p.~455]{chiangl:15}, the estimate \eqref{edmund_est} together with \cite[Theorem~2.1]{chiangf:08} implies the finite order case of Valiron's result: If a meromorphic function $f$ has finite order $\sigma$ and lower order $\lambda$ such that $\sigma-\lambda<1$, then $d(0,f(z))=d(0,f(z+c))$ since in this case the error term $O(r^{\sigma-1+\varepsilon})$ in \eqref{edmund_est} is small compared with $T(r,f)$ as $r\rightarrow\infty$ without any exceptional set. For example, the function $f(z)=e^z$ satisfies $\sigma(f)=\lambda(f)=1$ and $d(0,f(z))=d(0,f(z+c))=1$. The sharpness of Valiron's result is guaranteed by Gol'dberg's example above. Thus, for a meromorphic function $f$ of finite order $\sigma$ whose deficiency is dependent on the choice of origin, we have $\sigma-\lambda \geq 1$, and from the proof of Lemma~\ref{technical} and~\ref{danlogue} we know that, for a constant $C>1$ and a sufficiently small $\varepsilon>0$, the following inequality
    \begin{equation}\label{exceptional set0}
    T(\alpha(r+|c|),f) = T\left(r+|c|+\frac{r+|c|}{T(r+|c|,f)^{\varepsilon}},f\right) \leq  CT(r+|c|,f)\leq C^2T(r,f)
    \end{equation}
holds for all $r$ outside of an exceptional set of finite logarithmic measure.
If we apply \cite[Lemma~3.1]{halburdk:10JAMS} above with $\mu(r)=\log r$ to \eqref{exceptional set0} to remove this set, then we get from \eqref{hkt_est} that for some $\epsilon>0$,
    \begin{equation}\label{m relatioin}
    \begin{split}
    m\left( r,\frac{f(z+c)}{f(z)}\right)
    =o\left(\frac{T(e^{\epsilon}r,f)}{r^{1-\varepsilon/2}}\right)=o\left(\frac{r^{\sigma+\varepsilon/2}}{r^{1-\varepsilon/2}}\right)=o(r^{\sigma-1+\varepsilon}),
    \end{split}
    \end{equation}
where there exists an infinite sequence $r_n$ such that the error term in \eqref{m relatioin} satisfies $o(r_n^{\sigma-1+\varepsilon})\not=o(T(r_n,f))$ as $r_n\rightarrow\infty$.

\section{Discrete Borel type growth lemma extensions}\label{discrete_sec}

Osgood \cite{osgood:85} and, independently, Vojta \cite{vojta:87} observed that Nevanlinna's theory of value distribution and Diophantine approximation theory appear to be analogous on a deep level. Based on this observation Vojta composed a ``dictionary'' between the two theories. In this dictionary Nevanlinna theory appears to be ahead of Diophantine approximation theory in the sense that deep open conjectures in Diophantine approximation correspond to known classical results in Nevanlinna theory.

By replacing the continuous variable $r$ in Lemma~\ref{technical} with a sequence of positive numbers, we have the following discrete analogue of Lemma~\ref{technical}.
Our result is an extension of \cite[Lemma~8]{al-ghassanih:15} due to Al-Ghassani and Halburd, who applied their result to consider non-linear discrete equations with solutions $y_n\in\mathbb{Q}$ having slow height growth in terms of $n$. For clarity, in what follows we use $h(n)=h_n$ for a discrete sequence $\{h_n\}$.

\begin{lemma}\label{technicala}
Let $\{T_n\}_{n\geq n_0}$ $(n_0>0)$ be a non-decreasing sequence of positive numbers and let $s$ be a fixed positive integer. If
    \begin{equation}\label{assua}
    \limsup_{n\to\infty}\frac{h(n)h(nh(n)) \log T_n}{n}=\zeta,
    \end{equation}
where $\zeta\in[0,\infty)$ and $h(n)$ is an increasing sequence of positive numbers such that 
    \begin{equation*}
    \sum_{n=n_0}^\infty \frac{1}{nh(n)}<+\infty,
    \end{equation*}
then
   \begin{equation}\label{concla}
    T_{n+s} = T_n+(\zeta+o(1))\left(\frac{T_n}{h(n)}\right),
    \end{equation}
where $n$ runs to infinity outside of a set $E$ of finite discrete logarithmic
measure, i.e. $\sum_{n\in E}1/n<\infty$.
    \end{lemma}

\begin{proof}
For a fixed constant $\eta\in\mathbb{R}^+$ such that $\eta>\zeta$, assume that the set $F_{\eta}\subset\mathbb{N}$ defined by
    \begin{equation}\label{Fdefa}
    F_{\eta}=\left\{n\geq n_0:
    \frac{T_{n+s}-T_n}{T_n}\cdot h(n)\geq \eta
    \right\}
    \end{equation}
is of infinite discrete logarithmic measure, i.e. $\sum_{n\in F_\eta}1/n=\infty$. 
Let $r_0=\min(F_{\eta})$ and, for all $n\in \mathbb{N}$, set
$r_n=\min(F_\eta\cap F_{n-1})$, where $F_{n-1}=\{i\geq0:r_{n-1}+s+i\}$ is the set defined for all $i\in \mathbb{N}$. Then
the sequence $\{r_n\}_{n\in\mathbb{N}}$ satisfies $r_{n+1}-r_n\geq s$
for all $n\in\mathbb{Z}^+$, $F_\eta\subset \bigcup_{n=0}^\infty\{r_n,r_n+1,\ldots,r_n+s\}$ and
    \begin{equation}\label{assuinpr2a}
    \left(1+\frac{\eta}{h(r_n)}\right)T_{r_n}\leq T_{r_n+s}\leq T_{r_{n+1}}
    \end{equation}
for all $n\in\mathbb{Z}^+$. Suppose that there exists an integer $m\in\mathbb{Z}^+$
such that $r_n\geq n h(n)$ for all $n\geq m$. But
then,
    \begin{eqnarray*}
    \sum_{j\in F_\eta}\frac{1}{j} &\leq& \sum_{n=m}^\infty
    \sum_{k=[n h(n)]}^{[n h(n)]+s}\frac{1}{k}+O(1)\leq \sum_{n=m}^\infty\frac{s}{[n h(n)]}+O(1) \\ &\leq& \sum_{n=m}^\infty\frac{s}{nh(n)-1} +O(1)  <\infty,
    \end{eqnarray*}
where $[nh(n)]$ denotes the largest integer not exceeding $n h(n)$, which is a contradiction with the assumption $\sum_{j\in F_\eta}1/j=\infty$. Therefore, the sequence $\{r_n\}_{n\in\mathbb{Z}^+}$ has a subsequence $\{r_{n_j}\}_{j\in\mathbb{Z}^+}$ such that $r_{n_j}\leq n_jh(n_j)$ for all $j\in\mathbb{N}$.
Then, as in the proof of Lemma~\ref{technical}, by iterating~\eqref{assuinpr2a} along the sequence
$\{r_{n_j}\}_{j\in\mathbb{N}}$, it can be shown that $r_{n_j} \geq n_j$
for all $j\in\mathbb{Z}^+$ and that the superior limit of $h(n)h(nh(n)) \log T_n/n$ along the sequence $r_{n_j}$ is $\geq \eta$, which yields a contradiction to \eqref{assua}.
We omit those details. This implies that the discrete logarithmic measure of
$F_\eta$ defined by \eqref{Fdefa} must be finite, and so
    \begin{equation*}
    T_{n+s} = T_n+ (\zeta+o(1))\left(\frac{T_n}{h(n)}\right)
    \end{equation*}
for all $n$ outside of a set of finite discrete logarithmic measure.
Thus the assertion \eqref{concla} follows.
\end{proof}

Analogous to the continuous case, if we take $h(n)$ in Lemma~\ref{technicala} to be $h(n)=(\log n)^{1+\varepsilon}$, where $\varepsilon>0$, $n\geq n_0>e$, then
    \begin{equation*}
    \begin{split}
    \sum_{n=n_0}^\infty \frac{1}{n(\log n)^{1+\varepsilon}} &\leq \frac{1}{n_0h(n_0)}+\sum_{n=n_0+1}^\infty\int_{n-1}^n\frac{1}{t(\log t)^{1+\varepsilon}}dt\\  &\leq \frac{1}{n_0h(n_0)}+\int_{n_0}^{\infty}\frac{1}{t(\log t)^{1+\varepsilon}}dt <+\infty.
    \end{split}
    \end{equation*}
Then, if $\{T_n\}_{n\geq n_0}$ is a sequence of $n$ such that $\log T_n \leq n/(\log n)^{2+\nu}$, where $\nu>0$, we have
    \begin{equation*}
    \limsup_{n\to\infty}\frac{h(n)h(nh(n)) \log T_n}{n}=0,
    \end{equation*}
and so by Lemma~\ref{technicala} it follows that
    \begin{equation*}\label{ttt  sec3}
    T_{n+s} = T_n+ o\left(\frac{T_n}{(\log n)^{1+\varepsilon}}\right).
    \end{equation*}
Also, from the characteristic function $T(r,f)$ of the function $f(z)=\exp(e^z)$ we know that the constant $\zeta$ in Lemma~\ref{technicala} cannot be extended to $\infty$ since we may choose an infinite sequence of $r_n$ such that $r_n=n$.

\section{Malmquist's theorem for discrete equations}\label{dMalm_sec}

The algebraic entropy \cite{bellonv:99,hietarintav:98} of a discrete equation is defined as
    \begin{equation*}
    \lim_{j\to\infty}\frac{\log d_j}{j},
    \end{equation*}
where $d_j$ is the degree of the $j^{th}$ iterate of a discrete equation as a rational function of its initial conditions. If the algebraic entropy of a discrete equation is zero, then this is considered to be a strong sign of integrability of the equation. Consider as an example discrete equation
    \begin{equation}\label{d_eq}
    y_{n+1} = R(n,y_n)=\frac{P(n,y_n)}{Q(n,y_n)},
    \end{equation}
where $P(n,y_n)$ and $Q(n,y_n)$ are coprime polynomials in $y_n$ having rational coefficients in $\mathbb{Q}[n]$. For the autonomous version of \eqref{d_eq}, we have $d_j=[\deg_{y_0}(R)]^j$ and so the algebraic entropy is equal to $\log \deg_{y_0}(R)$ in this case. This implies that the algebraic entropy of \eqref{d_eq} is zero if and only if \eqref{d_eq} is the discrete Riccati equation. For a review on applications of algebraic entropy to the second order discrete equations, see \cite{grammaticoshrv:09}.

Halburd~\cite{halburd:05} has shown, assuming that the heights of the coefficients are small compared to the height of the solution, that the heights of iterates of the discrete equation \eqref{d_eq} over number fields grow exponentially, unless $\deg_{y_0}(R)=1$. Using this idea of Diophantine integrability, Al-Ghassani and Halburd obtained an extension of this result to the second order case by singling out the discrete Painlev\'e~II equation~\cite{al-ghassanih:15}.

As the final result of this study, we will apply Lemma~\ref{technicala} to give an improvement of Halburd's result on the first order discrete equations. Before stating the result, we need one more definition. Let $k$ be a number field, and let $\{y_n\}_{n\in\mathbb{N}}\subset k$ be a solution of \eqref{d_eq}, where the coefficients are in $k[n]$. For $x\in k$ we denote by $H(x)$ the \textit{height} and by $h(x)=\log H(x)$ \textit{the logarithmic height} of $x$. We say that $\{y_n\}_{n\in\mathbb{N}}$ is \textit{admissible} if the logarithmic heights of all coefficients of \eqref{d_eq} are of the growth $o(h(y_n))$ as $n\to\infty$ outside of an exceptional set of finite discrete logarithmic measure $\sum_{n\in E} 1/n < \infty $. This definition is an exact Diophantine analogue of the notion of admissible meromorphic solution of a difference equation in the spirit of Vojta's dictionary \cite{vojta:87}.

\begin{theorem}
Let $k$ be a number field, and let $\{y_n\}_{n\in\N}\subset k$ be an admissible solution of \eqref{d_eq}, where the coefficients are in $k[n]$. If
    \begin{equation*}
    \limsup_{n\to\infty} \frac{\log \sum_{k=1}^nh(y_k)}{n/(\log n)^{2+\nu}}=0
    \end{equation*}
for any $\nu> 0$, then $\deg_{y_0}(R)=1$.
\end{theorem}

\begin{proof}
By taking the logarithmic height of both sides of \eqref{d_eq}, it follows that
    \begin{equation*}
    h(y_{n+1}) = \deg_{y_0}(R) h(y_n) + o(h(y_n))
    \end{equation*}
as $n\to\infty$ outside of an exceptional set $E$ of finite discrete logarithmic measure. Therefore,
    \begin{equation*}
    \sum_{k=1}^{n+1}h(y_{k}) = \deg_{y_0}(R) \sum_{k=1}^{n} h(y_k) + o(h(y_n)),
    \end{equation*}
and so, by applying Lemma~\ref{technicala} with $T_n= \sum_{k=1}^{n} h(y_k)$, we have the assertion.
\end{proof}

\def\cprime{$'$}
\providecommand{\bysame}{\leavevmode\hbox to3em{\hrulefill}\thinspace}
\providecommand{\MR}{\relax\ifhmode\unskip\space\fi MR }
\providecommand{\MRhref}[2]{%
  \href{http://www.ams.org/mathscinet-getitem?mr=#1}{#2}
}
\providecommand{\href}[2]{#2}

%

\end{document}